\documentclass[12pt]{article}

\usepackage[dvipdf]{graphicx}
\usepackage[latin1]{inputenc}
\usepackage{latexsym}
\usepackage{amsmath,amssymb,amsfonts,amsthm}
\usepackage{xcolor}
\usepackage{mathrsfs}
\usepackage{bbm}
\usepackage{bm}
\usepackage{enumitem}

\usepackage[numbers,comma,sort]{natbib}

\definecolor{db}{RGB}{0, 0, 130}
\definecolor{wildstrawberry}{rgb}{1.0, 0.26, 0.64}

\usepackage[colorlinks=true,citecolor=db,linkcolor=black,urlcolor=blue,pdfstartview=FitH]{hyperref}

\usepackage{pdfsync}

\definecolor{rp}{rgb}{0.25, 0, 0.75}
\definecolor{dg}{rgb}{0, 0.5, 0}

\newcommand{\R}{\mathbb{R}}

\newcommand{\N}{\mathbb{N}}
\newcommand{\EE}{\mathbb{E}}

\makeatletter
\newcommand{\customlabel}[2]{%
   \protected@write \@auxout {}{\string\newlabel {#1}{{#2}{\thepage}{#2}{#1}{}}}%
   \hypertarget{#1}{#2\hspace{-0.14cm}}
}
\makeatother

\numberwithin{equation}{section}

\newtheorem{theorem}{Theorem}[section]

\newtheorem{definition}[theorem]{Definition}

\newtheorem{lemma}[theorem]{Lemma}

\newtheorem{remark}[theorem]{Remark}

\author{ Jesus Correa \footnote{Departamento de Matem\'atica, Universidade Estadual de Campinas, Brazil. \texttt{colivera@ime.unicamp.br}.}
 \and    Christian Olivera \footnote{Departamento de Matem\'atica, Universidade Estadual de Campinas, Brazil. \texttt{colivera@ime.unicamp.br}.} \and 
}

\title{ \Large{\textbf{From Stochastic particle  Systems to Stochastic Compressible  Euler Equation}}}
\begin{document}

\maketitle

\begin{abstract}
We consider  stochastic systems of interacting particles with noise  through a potential whose range is large in comparison with the typical distance between neighbouring particles.
It is shown that the empirical measures associated to the position and velocity of the system converge to the solutions of  stochastic compressible Euler equations  in the limit as the particle number tends to infinity. Moreover, we quantify the  distance  between particles and the limit in  suitable Sobolev norm.  
\end{abstract}

\date{}

\maketitle

\noindent \textit{ {\bf Key words and phrases:} 
Stochastic Differential Equation, Compressible  Euler  Equation, Particle systems,  Ito-Kunita-Wentzell formula.}

\vspace{0.3cm} \noindent {\bf MSC2010 subject classification:} 60H15, 
 35R60, 60H30, 35Q31
. 

%
%
%
%
%
%
%

Data sharing not applicable to this article as no datasets were generated or analysed during the current study
\section{Introduction}

In this work we analyse the evolution of  $N$   particles in  $\R^{d}$  where the position $X_{k}^{N}$ verifies for $k=1,\cdots, N$
\begin{align}\label{xk}
d^{2}X_{t}^{k,N}=-\frac{1}{N}\sum_{l=1}^{N}\nabla\phi_{N}\left(X_{t}^{k,N}-X_{t}^{l,N}\right)dt
+ \sigma\left(X_{t}^{k,N} \right)  \frac{dX^{k,N}}{dt}  \circ dB_{t} 
\end{align}

where   $\{(B_{t}^{q})_{t\in[0,T]}, \; q=1,\cdots,d\}$ is a  standard $\R^d$-valued Brownian motions
and the integration in the sense of Stratonovich. The interaction potential  $\phi_{N}$  is obtained from a function  $\phi_{1}$  by the scaling

 \begin{equation}\label{molli}
 \phi_{N}(x)=N^{\beta}\phi_{1}(N^{\beta/d}x),\   \beta\in(0,1).
\end{equation}

It will turn out that the typical distance between neighbouring particles is of order
$N^{1/d}$ as $N\rightarrow \infty$. Hence, the scaling (\ref{molli}) describes a physical situation in which
the interaction has long range. We shall always assume that $\phi_{1}$ is symmetric and
sufficiently smooth.

Our aim is the study of the asymptotic as $N\rightarrow\infty$ of  the time evolution
of the whole system of all particles. Therefore, we investigate the empirical processes

\begin{eqnarray}\label{empp}
S_{t}^{N}:=\frac{1}{N}\sum_{k=1}^{N}\delta_{X_{t}^{k,N}},
\end{eqnarray}

\begin{eqnarray}\label{empp}
 V_{t}^{N}:=\frac{1}{N}\sum_{k=1}^{N}V_{t}^{k,N}\delta_{X_{t}^{k,N}}\qquad k=1,\cdots, N
\end{eqnarray}

where   $\frac{d}{dt} X_{t}^{k,N}=V_{t}^{k,N}$   is the velocity of the $kth$ particle, and $\delta_{a}$, denotes the
Dirae measure at $a$. The measures $S_{t}^{N}$ and $V_{t}^{N}$ determine the distribution
of the positions and the velocities in the $Nth$ system. 
We shall show that $S_{t}^{N}$ and $V_{t}^{N}$ converge as $N\rightarrow\infty$ to solutions of the
continuity equation and the stochastic Euler equation, respectively. 
 The addition of stochastic terms to the governing equations is commonly used to account for
empirical, numerical, and  uncertainties in applications ranging from climatology to turbulence
theory. The stochastically forced system of the compressible  Euler system describing the time
evolution of the mass density  $\varrho$  and the bulk velocity $\upsilon$ of a fluid,

\begin{align}\label{eq1}
\left\{
\begin{array}{lc}
d\varrho =-\operatorname{div}_{x}(\varrho\upsilon)dt
\\[0.3cm]
d(\varrho\upsilon_{q})=-\left(\nabla^{q}p+\operatorname{div}_{x}(\varrho\upsilon_{q}\upsilon)\right)dt 
+ \varrho \sigma_{q}( x) \upsilon_{q} \circ dB_{t}^{q}\qquad q=1,\cdots, d.
\end{array}
\right.
\end{align}

Applying   to  $\varrho\upsilon_{q}$  the  Ito formula for the product of  semimartingale 
and  using the equation that   verifies  $\varrho$,  we can show  that the system  (\ref{eq1}) is equivalent to

\bigskip
\begin{align}\label{eq2}
\left\{
\begin{array}{lc}
d\varrho =-\operatorname{div}_{x}(\varrho\upsilon)dt
\\[0.3cm]
d(\upsilon_{q})=-\left(\frac{1}{\varrho} \nabla^{q}p+\upsilon  \nabla  \upsilon_{q} \right) dt+    \sigma_{q}( x) \upsilon_{q}\circ dB_{t}^{q}\qquad q=1,\cdots, d
\end{array}
\right.
\end{align}

where the pressure is $p=\frac{1}{2}\varrho^{2}.$   
We present the first macroscopic derivation of the stochastic compressible Euler equation. Our proof involves several tools, for example :   Ito-Kunita-Wentzell formula, commutator estimates, Taylor expansion   and some ideas in  the paper \cite{Oes} where the  same equation   without noise  is considered. For related works in the deterministic setting see 
\cite{Carillo}, \cite{Dil} and \cite{Franz}.

The problem of solvability of the stochastic compressible Euler system (\ref{eq2}) is very
challenging with only a few results available. In space dimension one, \cite{Vove}
proved existence of entropy solutions. The only available results in higher space dimensions concern the local well-posedness of
strong solutions, see \cite{Breit}, \cite{Gho}, \cite{Kim},  and references.

\subsection{Related works.}

Mean field limits of large particle systems and fluid dynamics equation 
is an active field of research, hence it is out of the scope of this article to make a full review of it. Instead, we refer to  \cite{Carillo, Maurelli,  FHM, Gui,JabinWang, Marchioro, Meleard3, Osada, Serfaty}   and the references therein for the recent developments. In the following we will focus only on the results of moderately interacting systems :  in  \cite{Oelschlager84} Oelschlager introduced and studied moderately interacting particle systems which are used to obtain local non-linear partial differential equations. The article \cite{Oelschlager84}  is part of a series of works by the author on this subject (see also \cite{Oes2,Oes4}), the convergence results was improved  in \cite{JourdainMeleard} and \cite{Meleard}.  In the recent paper \cite{FlandoliLeimbachOlivera} was  developed a semigroup approach which enables them to show uniform convergence of mollified empirical measures,  see \cite{FlandoliLeocata,Leocata, Live, FlandoliOliveraSimon,Simon} for further applications of
this method.  Recently, in \cite{Pisa} and \cite{ORT} the authors developed new quantitative  estimates for classes of moderately interacting particle systems,  by using a semi-group approach, this approach 
was improvements  in \cite{Hao},  \cite{Knorst} and   \cite{Simon2}. About more advances in  moderate particle systems,  see for instance  \cite{Ansgar}, \cite{Carillo2},  \cite{Chen}, \cite{correa}, \cite{correa2},  \cite{Ansgar2}, and \cite{Steve}.

The literature on particle with  common noise   remains limited ,  for systems with uniformly Lipschitz interaction coefficients \cite{Cogui} established conditional propagation of chaos,  the entropy method has recently been explored for systems with common noise, as shown in \cite{Shao} for incompressible  the Navier-Stokes equations, \cite{Chen2} for the Hegselmann-Krause model, and \cite{Niko} for mean-field systems with bounded kernels. For more results , see \cite{Maurelli}, \cite{correa}, \cite{Knorst2},   \cite{Kotelenez} and \cite{Rose}.

\subsection{Heuristic deduction}
By Ito formula we have 
\begin{align}\label{eq1d.N-S}
\langle S_{t}^{N},f\rangle=&\langle S_{0}^{N},f\rangle+\int_{0}^{t}\langle  S_{s}^{N} ,\nu^{N}(\cdot,s)\nabla f\rangle ds
\end{align}
and 
\begin{align*}
\langle S_{t}^{N},f \nu_{q}^{N}(\cdot,t)\rangle=&\langle S_{0}^{N},f\nu_{q}^{N}(\cdot,0)  \rangle -
\int_{0}^{t}\langle S_{s}^{N},f\nabla^{q} \left(S_{s}^{N}\ast\phi_{N}\right)\rangle ds\\
&+\int_{0}^{t}\langle S_{s}^{N},\nu_{q}^{N}(\cdot,s) \nabla f\cdot \nu^{N}(\cdot,s)\rangle ds+\int_{0}^{t}\langle S_{s}^{N},f \sigma(\cdot)\nu_{q}^{N}   \rangle\circ dB_{s}^{q}\\
&\qquad\qquad\qquad\qquad\qquad\qquad\qquad\qquad\qquad\qquad\qquad q=1,\cdots,d
\end{align*}

The system (\ref{xk}) is equivalent to the  system:
\begin{align}\label{VN}
\left\{
\begin{array}{lc}
d X_{t}^{k,N}=V_{t}^{k,N}dt
\\[0.3cm]
d V_{t}^{k,N}=-\nabla\left(S_{t}^{N}\ast\phi_{N}\right)(X_{t}^{k,N})dt+  \sigma (X_{t}^{k,N})  V_{t}^{k,N} \circ d B_{t}\quad k=1,\cdots, N
\end{array}
\right.
\end{align}

\vspace{0.50cm}

Applying the Ito formula  for  the product  
\begin{align*}
d(V_{t,q}^{k,N}f(X_{t}^{k,N}))=&f(X_{t}^{k,N})\circ d(V_{t,q}^{k,N})+V_{t,q}^{k,N}\circ d(f(X_{t}^{k,N}))\\
    =&-f(X_{t}^{k,N})\nabla^{q}\left( S_{t}^{N}\ast\phi_{N}\right)(X_{t}^{k,N})))dt\\
    &+f(X_{t}^{k,N})\sigma (X_{t}^{k,N})  V_{t,q}^{k,N} \circ d B_{t}^{q}+ V_{t,q}^{k,N}\nabla f(X_{t}^{k,N})\cdot V_{t}^{k,N} dt
\end{align*}
or , equivalently,
\begin{align*}
&f(X_{t}^{k,N})V_{t,q}^{k,N}-f(X_{0}^{k,N})V_{0,q}{k,N}=-\int_{0}^{t}f(X_{s}^{k,N})\nabla^{j}\left( S_{t}^{N}\ast\phi_{N}\right)(X_{s}^{k,N})ds\\
& \quad+\int_{0}^{t}f(X_{s}^{k,N})\sigma (X_{s}^{k,N})  V_{s,q}^{k,N} \circ d B_{s}^{q}+\int_{0}^{t}V_{s,q}^{k,N}\nabla f(X_{s}^{k,N})\cdot V_{s}^{k,N}ds.
\end{align*}
Then we get 
\begin{align}\label{EnuN-S}
&\frac{1}{N}\sum_{k=1}^{N}f(X_{t}^{k,N})V_{t,q}^{k,N}-f(X_{0}^{k,N})V_{0,q}{k,N}=-\int_{0}^{t}\frac{1}{N}\sum_{k=1}^{N}f(X_{s}^{k,N})\nabla^{j}\left( S_{t}^{N}\ast\phi_{N}\right)(X_{s}^{k,N})ds\nonumber\\
&+\int_{0}^{t}\frac{1}{N}\sum_{k=1}^{N}f(X_{s}^{k,N})\sigma (X_{s}^{k,N})  V_{s,q}^{k,N} \circ d B_{s}^{q}+\int_{0}^{t}\frac{1}{N}\sum_{k=1}^{N}V_{s,q}^{k,N}\nabla f(X_{s}^{k,N})\cdot V_{s}^{k,N}ds.
\end{align}
Let us now introduce formally  an $\R^{d}$-valued  function $\nu^{N}$
which satisfies $\nu^{N}(X_{t}^{k,N}, t) = V_{t}^{k,N}$, $ k = 1,\cdots,N$. Obviously, such a
function can be defined only if different particles cannot occupy the same position
at the same instant.

From \eqref{EnuN-S} we deduce
\begin{align}\label{eq2d.N-S}
\langle S_{t}^{N},f \nu_{q}^{N}(\cdot,t)\rangle=&\langle S_{0}^{N},f\nu_{q}^{N}(\cdot,0)  \rangle -
\int_{0}^{t}\langle S_{s}^{N},f\nabla^{q} \left(S_{s}^{N}\ast\phi_{N}\right)\rangle ds\nonumber\\
&+\int_{0}^{t}\langle S_{s}^{N},\nu_{q}^{N}(\cdot,s) \nabla f\cdot \nu^{N}(\cdot,s)\rangle ds+\int_{0}^{t}\langle S_{s}^{N},f \sigma(\cdot)\nu_{q}^{N}   \rangle\circ dB_{s}^{q}\nonumber\\
&\qquad\qquad\qquad\qquad\qquad\qquad\qquad\qquad\qquad\qquad\qquad\quad q=1,\cdots,d.\nonumber\\
\end{align}

Thus, by taking the limit when $N\to\infty$ in the equations \eqref{eq1d.N-S} and \eqref{eq2d.N-S} as $N\to\infty$, we can formally deduce  the system of equations (\ref{eq1}).

\section{ Definitions,  preliminaries and hypothesis.}

\subsection{Space of functions }
For any  $s\in \R$,  we denote by
$H^{s}(  \mathbb{R}^{d})  $ the \emph{Bessel potential space}
\[H^{s}(  \mathbb{R}^{d}):= \Big\{ u \text{ tempered distribution; }  \;   \big(1+|\cdot|^{2}\big)^{\frac{s}{2}}\; \mathcal{F} u(\cdot)  \in  L^2(\mathbb R^d)\Big\},  \]
where $\mathcal Fu$ denotes the \emph{Fourier transform} of $u$. This space is endowed with the norm
 \begin{equation*}
 \| u \|_{s} = \left\| \big(1+|\cdot|^{2}\big)^{\frac{s}{2}
}\; \mathcal{F} u(\cdot)  \right\|_{L^2(\R^d)}. 
 \end{equation*}

We denoted $C_{b}(\R^{d}; \R^{d})$  is the space of bounded continuous $\R^{d}$-valued functions on $\R^{d}$.
Here we used the brackets as an abbreviation for an integral over $\R^d$.

\subsection{Stochastic calculus }
First, through of this paper, we fix a stochastic basis with a
$d$-dimensional Brownian motion $\big( \Omega, \mathcal{F}, \{
\mathcal{F}_t: t \in [0,T] \}, \mathbb{P}, (B_{t}) \big)$. Then, we recall to help the intuition, the following definitions 

$$
\begin{aligned}
\text{Ito:}&  \ \int_{0}^{t} X_s dB_s=
\lim_{n    \rightarrow \infty}   \sum_{t_i\in \pi_n, t_i\leq t}  X_{t_i}    (B_{t_{i+1} \wedge t} - B_{t_i}),
\\[5pt]
\text{Stratonovich:}&  \ \int_{0}^{t} X_s  \circ dB_s=
\lim_{n    \rightarrow \infty}   \sum_{t_i\in \pi_n, t_i\leq t} \frac{ (X_{t_i \wedge t   } + X_{t_i} ) }{2} (B_{t_{i+1} \wedge t} - B_{t_i}),
\\[5pt]
\text{Covariation:}& \ [X, Y ]_t =
\lim_{n    \rightarrow \infty}   \sum_{t_i\in \pi_n, t_i\leq t} (X_{t_i \wedge t   } - X_{t_i} )  (Y_{t_{i+1} \wedge t} - Y_{t_i}),
\end{aligned}
$$
where $\pi_n$ is a sequence of finite partitions of $ [0, T ]$ with size $ |\pi_n| \rightarrow 0$ and
elements $0 = t_0 < t_1 < \ldots  $. The limits are in probability, uniformly in time
on compact intervals. Details about these facts can be found in Kunita 
\cite{Ku2}.

 Now, we recall  the Ito-Kunita-Wentzell formula, see  Theorem 8.3 of \cite{Ku2}.
We consider $X(t,x, \omega)$ be a continuous  $C^{3}$-process and 
continuous  $C^{2}$-semimartingale , for $x\in \R^{d}$, and $t\in [0,T]$ of the form 
\begin{equation*}
 X(t,x,\omega)=X_{0}(x) + \int_{0}^{t} f(s,x,\omega) ds  + \int_{0}^{t} g(s,x,\omega) \circ dB_s.   
\end{equation*}
If $Y(t,\omega)$   is  a continuous semimartingale, then $X(t,Y_{t},\omega)$
is a continuous semimartingale and the following formula holds 
\begin{align*}
 X(t,Y_{t},\omega)=&X_{0}(Y_{0}) + \int_{0}^{t} f(s,Y_{s},\omega) ds  + \int_{0}^{t} g(s,Y_{s},\omega) \circ dB_s\\
 &+ \int_{0}^{t} (\nabla_{x}X)(s,Y_{s},\omega)\circ dY_{s}.
\end{align*}

\subsection{Definition of solution }

  \begin{definition}
	Let $(\varrho_{0},\upsilon_{0})$ be a $H^{s}\left(\R^{d}\right) \times H^{s}\left(\R^{d}\right)$, $s>\frac{d}{2}+2$ and $\varrho_{0}>0$.
Then  $(\varrho, \upsilon, \tau)$   is called a solution of (\ref{eq2}) if the following conditions are satisfied

\begin{enumerate}
    \item[(i)] $\left(\varrho, \upsilon\right)$ is a $H^{s}\left(\R^{d}\right) \times H^{s}\left(\R^{d}\right)$-valued right-continuous progressively measurable process;
\item[(ii)] $\tau$ is a stopping time with respect to $\left(\mathfrak{F}_{t}\right)$ such that $\mathbb{P}$-a.s.
\begin{equation*}
 \tau(w)=\lim _{m \rightarrow \infty} \tau_{m}(w)   
\end{equation*}
where
\begin{equation*}
  \tau_{m}=\inf \left\{0 \leq t<\infty:\left\|\left(\varrho, \upsilon\right)(t)\right\|_{H^{s}(\R^{d})} \geq m\right\}  
\end{equation*}
with the convention that $\tau_{m}=\infty$ if the set above is empty;
\item[(iii)] there holds $\mathbb{P}$-a.s.
\begin{equation*}
   (\varrho, u)   \in C\left([0, \tau_{m}(\omega)] ; H^{s}\left(\R^{d}\right) \times H^{s}\left(\R^{d}\right)\right)   
\end{equation*}
as well as
\begin{align*}
&\varrho\left(t \wedge \tau_{m}\right)=\varrho_{0}-\int_{0}^{t \wedge \tau_{m}} \operatorname{div}_{x}(\varrho\upsilon)\mathrm{d} s, \\
&\upsilon_{q}\left(t \wedge \tau_{m}\right)=\upsilon_{q,0}-\int_{0}^{t \wedge \tau_{m}} \left(\frac{1}{\varrho} \nabla^{q}p+\upsilon  \nabla  \upsilon_{q} \right)\mathrm{d} s+\int_{0}^{t \wedge \tau_{m}} \sigma_{q} \upsilon_{q} \circ dB_{t}^{q}
\end{align*}
for $q=1,\cdots, d$, all $t \in[0, T]$ and all $m\geq 1$.
\end{enumerate}

  \end{definition}
	
\begin{remark}
	On existence and uniqueness  of  solutions for the equations (\ref{eq1})-(\ref{eq2}) we refer to
	 \cite{Breit} and \cite{Kim}.
	\end{remark}

\begin{remark} We recall that by the method of characteristics $\varrho= \varrho_{0}(X_{t}^{-1})$, where $X_{t}$ is the flow 
with drift $\upsilon$, that is, $X_{t}=x+ \int_{0}^{t}\upsilon(t,X_{t}) \  dt$. Therefore,  if $\varrho_{0}>0$  then $\varrho>0$. 
\end{remark}

\begin{remark} We observe that by Sobolev embedding \\  $(\varrho, \upsilon)   \in C\left([0, \tau_{m}(\omega)] ;  C_{b}^{2} \cap L^{1}(\R^{d})  \times   C_{b}^{2}(\R^{d}) \right)$. Then the solutions are classical in the space variable.  The existence of global solutions in time  and explosion criterion are   difficult problems even in the deterministic case.
The stochastic case will be a topic of future research.
\end{remark}

\subsection{ Technical hypothesis}

Next, we suppose that $\phi_{1}$ can be written as a convolution product 

 $$\phi_{1}=\phi_{1}^{r}\ast\phi_{1}^{r}$$

 where  $\phi_{1}^{r}\in C_{b}^{2}(\R^{d})$  and it is symmetric probability density.

\begin{definition}
For all  $q\in\{1,\cdots,d\}$  we define the function 
\begin{align*}
U_{1;\alpha}^{q}:\R^{d}&\longrightarrow\R\\
x&\longmapsto(-1)^{1+|\alpha|} \dfrac{x^{\alpha}}{\alpha!} \nabla^{q}\phi_{1}^{r}(x)
\end{align*}
where  $0\leq|\alpha|\leq L+1$ with $L:=\left[\frac{d+2}{2}\right]$.
\end{definition}

We assume that functions $\phi_{1}^{r}$ and $U_{1;\alpha}^{q}$ satisfy

\begin{align}\label{c1}
|\phi_{1}^{r}(x)|\leq\frac{C}{1+|x|^{d+2}}\qquad |x|\geq1,
\end{align}
\begin{align}\label{cotawildeu}
|\widetilde{U^{q}_{1;\alpha}}(\lambda)|\leq C|\widetilde{\phi_{1}^{r}}(\lambda)|,\qquad\text{para } 1\leq|\alpha|\leq L ,
\end{align}
and 
\begin{align}\label{cotauj}
|U^{q}_{1;\alpha}(x)|\leq C \left(\frac{1}{1+|x|^{d+1}}\right)^{1/2},\qquad\text{para } |\alpha|=L+1\ .
\end{align}
On the existence of the function  $\phi_{1}(x)$ we refer to 
   \cite{Oes3}. We also assume that ${\sigma\in C_{b}^{1}( \R^{d}, \R^{d})}$.

\section{Result}

First we present a simple lemma. 

\begin{lemma}
We assume \eqref{c1}, then there exist $C^{\prime}>0$ such that
\begin{align}\label{cotafNbeta}
\Vert f-(f\ast\phi_{N}^{r})\Vert_{\infty}\leq&C^{\prime}N^{-\beta/d}\Vert\nabla f\Vert_{\infty}\qquad\forall f\in C_{b}^{1}(\R^{d})
\end{align}
\end{lemma}

\begin{proof}
\begin{align*}
|f(x)-(f\ast\phi_{N}^{r})(x)|=&\left\vert\int_{\R^{d}}(f(x)-f(x-y))\phi_{N}^{r}(y)\thinspace dy \right\vert\nonumber\\
\leq&\Vert\nabla f\Vert_{\infty}\int_{\R^{d}}\phi_{N}^{r}(y)|y|\thinspace dy\nonumber\\
=&\Vert\nabla f\Vert_{\infty}\int_{\R^{d}}N^{\beta}\phi_{1}^{r}(N^{\beta/d}y)|y|\thinspace dy\nonumber\\
=&\Vert\nabla f\Vert_{\infty}\int_{\R^{d}}\phi_{1}^{r}(x))|N^{-\beta/d}x|\thinspace dx\\
\leq&C_{0}N^{-\beta/d}\Vert\nabla f\Vert_{\infty}\nonumber
\end{align*}
\end{proof}

We define
 
\begin{align}\label{QN}
Q_{t}^{N}:=\frac{1}{N}\sum_{k=1}^{N}|V_{t}^{k,N}-\upsilon(X_{t}^{k,N},t)|^{2} +\Vert S_{t}^{N}\ast\phi_{N}^{r}-\varrho(.,t)\Vert_{0}^{2}.
\end{align}

\begin{theorem}\label{tomeuler}
Let $s\geq\frac{d}{2}+3$ and $\alpha>\frac{d}{2}+1$.  We assume \eqref{molli}, \eqref{c1}-\eqref{cotauj}.
Then there exist  $C_{m,T}>0$ such that 
\begin{equation*}
\EE Q_{t\wedge \tau_{m}}^{N}\leq C_{m}(\EE Q_{0}^{N}+N^{-\beta/d})\quad\text{para todo } N\in\N,
\end{equation*}
\begin{equation*}
\EE  \|S_{t\wedge \tau_{m}}^{N}- \varrho_{\wedge \tau_{m}}\|_{-\alpha}^{2} \leq C_{m}(\EE Q_{0}^{N}+N^{-\beta/d}) \quad\text{for all } N\in\N,
\end{equation*}
and
\begin{equation*}
\EE   \|V_{t\wedge \tau_{m}}^{N}- (\varrho \upsilon)_{t\wedge \tau_{m}}\|_{-\alpha}^{2} \leq C_{m}(\EE  Q_{N}(0)+N^{-\beta/d}) \quad\text{for all } N\in\N,
\end{equation*}

\end{theorem}

\begin{remark}
  We observe by hypotheses  $s> d/2  + 3 $, thus   by 
 by Sobolev embedding $(\varrho, \upsilon)   \in C\left([0, \tau_{m}(\omega)] ;  C_{b}^{3} \cap L^{1}(\R^{d})  \times   C_{b}^{3}(\R^{d}) \right)$. 
 Then for all $m$ there exist a constant  $C_{m}$ such that $\| \varrho \|_{C([0,\tau_{m}], C_{b}^{3}(\R^{d}) )}\leq C_{m} $ and $\| \varrho\|_{C([0,\tau_{m}], C_{b}^{3}(\R^{d}) )}\leq C_{m} $, hence
 $\varrho_{t\wedge \tau_{m}} $ and  $\upsilon_{t\wedge \tau_{m}}$ are   $C^{3}$-semimartingales. Therefore,  we can apply Ito-Kunita-Wentzell formula to  $ \upsilon_{t\wedge \tau_{m}} (X_{t}^{k,N})$.
 
\end{remark}

\begin{proof}

We denoted $V_{t,q}^{k,N},\upsilon_{q} $ the $q$-coordinate of $V_{t}^{k,N}$ and $\upsilon $
respectively. During the proof we use conveniently   the  Einstein's summation convention.
In order not to load the notation during a large part of the proof, we use $t$ for $t\wedge \tau_{m}$. First we observe that 
\begin{align*}
Q_{t}^{N}=&\dfrac{1}{N}\sum_{k=1}^{N}\vert V_{t}^{k,N}\vert^{2}-\dfrac{2}{N}\sum_{k=1}^{N} V_{t}^{k,N}\cdot\upsilon(X_{t}^{k,N},t)+\dfrac{1}{N}\sum_{k=1}^{N}\vert\upsilon(X_{t}^{k,N},t)\vert^{2}+\Vert S_{t}^{N}\ast\phi_{N}^{r}\Vert_{0}^{2}\\
&-2\langle  S_{t}^{N}\ast\phi_{N}^{r},\varrho(.,t)\rangle+\Vert\varrho(.,t)\Vert_{0}^{2}\\
=&\dfrac{1}{N}\sum_{k=1}^{N}\vert V_{t}^{k,N}\vert^{2}-\dfrac{2}{N}\sum_{k=1}^{N} V_{t}^{k,N}\cdot\upsilon(X_{t}^{k,N},t)+\dfrac{1}{N}\sum_{k=1}^{N}\vert\upsilon(X_{t}^{k,N},t)\vert^{2}\\
&+\dfrac{1}{N^{2}}\sum_{k,l=1}^{N}\phi_{N}(X_{t}^{k,N}-X_{t}^{l,N})-\dfrac{2}{N}\sum_{k=1}^{N}\left(\varrho(.,t)\ast\phi_{N}^{r}\right)(X_{t}^{k,N})+\Vert\varrho(.,t)\Vert_{0}^{2}.
\end{align*}

We will calculate  the  Stratonovich differential of each term in the sum of $Q_{t}^{N}$. 

By Ito formula we have 

\begin{align*}
d\left(\dfrac{1}{N}\sum_{k=1}^{N}\vert V_{t}^{k,N}\vert^{2}\right)=&
\dfrac{1}{N}\sum_{k=1}^{N}d\left(\vert V_{t}^{k,N}\vert^{2}\right)=\dfrac{1}{N}\sum_{k=1}^{N}2V_{t,q}^{k,N}\circ d(V_{t,q}^{k,N})\\
=&-\dfrac{2}{N}\sum_{k=1}^{N}V_{t,q}^{k,N}\nabla^{q}\left(S_{t}^{N}\ast\phi_{N}\right)(X_{t}^{k,N})dt\\
&+\dfrac{2}{N}\sum_{k=1}^{N}V_{t,q}^{k,N}  \sigma_{q}(X_{t}^{k,N}) V_{t,q}^{k,N}   \circ d B_{t}^{q}.
\end{align*}
Applying Ito formula  for the product and
 Ito-Kunita-Wentzell formula  we deduce 
\begin{align}\label{went}
&d\left(-\dfrac{2}{N}\sum_{k=1}^{N} V_{t}^{k,N}\cdot\upsilon(X_{t}^{k,N},t)\right)
=-\dfrac{2}{N}\sum_{k=1}^{N}d[ V_{t,q}^{k,N}\upsilon_{q}(X_{t}^{k,N},t)]\nonumber\\
=&-\dfrac{2}{N}\sum_{k=1}^{N}\upsilon_{q}(X_{t}^{k,N},t)\circ d(V_{t,q}^{k,N})-\dfrac{2}{N}\sum_{k=1}^{N}V_{t,q}^{k,N}\circ d(\upsilon_{q}(X_{t}^{k,N},t))\nonumber\\
=&\dfrac{2}{N}\sum_{k=1}^{N}\upsilon(X_{t}^{k,N},t)\cdot\nabla\left(S_{t}^{N}\ast\phi_{N}\right)(X_{t}^{k,N})dt+\dfrac{2}{N}\sum_{k=1}^{N}V_{t,q}^{k,N}\upsilon(X_{t}^{k,N},t)\cdot\nabla\upsilon_{q}(X_{t}^{k,N},t)dt\nonumber\\
&+\dfrac{2}{N}\sum_{k=1}^{N}V_{t,q}^{k,N}\nabla^{q}\varrho(X_{t}^{k,N},t)dt-\dfrac{2}{N}\sum_{k=1}^{N}V_{t,q}^{k,N}\nabla\upsilon_{q}(X_{t}^{k,N},t)\cdot V_{t}^{k,N}dt\nonumber\\
&-\dfrac{2}{N}\sum_{k=1}^{N}\upsilon_{q}(X_{t}^{k,N},t) \sigma_{q}(X_{t}^{k,N})  V_{t,q}^{k,N} \circ dB_{t}^{q}- \dfrac{2}{N}\sum_{k=1}^{N}\ V_{t,q}^{k,N}  \sigma_{q}(X_{t}^{k,N}) \upsilon_{q}(X_{t}^{k,N},t) \circ dB_{t}^{q}.
\end{align}
Applying Ito formula  for the product and
 Ito-Kunita-Wentzell formula  we obtain

\begin{align*}
&d\left(\dfrac{1}{N}\sum_{k=1}^{N}\vert\upsilon(X_{t}^{k,N},t)\vert^{2}\right)=\dfrac{2}{N}\sum_{k=1}^{N}\upsilon_{q}(X_{t}^{k,N},t)\circ d(\upsilon_{q}(X_{t}^{k,N},t))\\
=&-\dfrac{2}{N}\sum_{k=1}^{N}\upsilon_{q}(X_{t}^{k,N},t)\upsilon(X_{t}^{k,N},t)\cdot\nabla\upsilon_{q}(X_{t}^{k,N},t)-\dfrac{2}{N}\sum_{k=1}^{N}\upsilon_{q}(X_{t}^{k,N},t)\nabla^{q}\varrho(X_{t}^{k,N},t)dt\\
&+\dfrac{2}{N}\sum_{k=1}^{N}\upsilon_{q}(X_{t}^{k,N},t) \sigma_{q}(X_{t}^{k,N})  \upsilon_{q}(X_{t}^{k,N},t) \circ dB_{t}^{q}\\
&+\dfrac{2}{N}\sum_{k=1}^{N}\upsilon_{q}(X_{t}^{k,N},t)\nabla \upsilon_{q}(X_{t}^{k,N},t)\cdot V_{N}^{k}(t)dt.
\end{align*}
By Leibniz rule and simple calculation we have 
\begin{align*}
d\left(\dfrac{1}{N^{2}}\sum_{k,l=1}^{N}\phi_{N}(X_{t}^{k,N}-X_{t}^{l,N})\right)=&\dfrac{1}{N^{2}}\sum_{k,l=1}^{N}d\left(\phi_{N}(X_{t}^{k,N}-X_{t}^{l,N})\right)\\
=&\dfrac{1}{N^{2}}\sum_{k,l=1}^{N}\nabla^{q}\phi_{N}(X_{t}^{k,N}-X_{t}^{l,N})\circ d \left(X_{t,q}^{k,N}-X_{t,q}^{l,N} \right)\\
=&\dfrac{1}{N^{2}}\sum_{k,l=1}^{N}\nabla^{q}\phi_{N}(X_{t}^{k,N}-X_{t}^{l,N}) \left(V_{t,q}^{k,N}-V_{t,q}^{l,N} \right)dt\\
=&\dfrac{1}{N}\sum_{k=1}^{N}\left(\dfrac{1}{N}\sum_{l=1}^{N}\nabla^{q}\phi_{N}(X_{t}^{k,N}-X_{t}^{l,N})\right)\ V_{t,q}^{k,N}dt\\
\end{align*}
\begin{align*}
\qquad\qquad\qquad\qquad\qquad\qquad\quad &-\dfrac{1}{N}\sum_{l=1}^{N}\left(\dfrac{1}{N}\sum_{k=1}^{N}\nabla^{q}\phi_{N}(X_{t}^{k,N}-X_{t}^{l,N})\right) V_{t,q}^{l,N}dt \\
=&\dfrac{1}{N}\sum_{k=1}^{N}\left(\dfrac{1}{N}\sum_{l=1}^{N}\nabla^{q}\phi_{N}(X_{t}^{k,N}-X_{t}^{l,K})\right) V_{t,q}^{k,N}dt\\
 &-\dfrac{1}{N}\sum_{l=1}^{N}\left(-\dfrac{1}{N}\sum_{k=1}^{N}\nabla^{q}\phi_{N}(X_{t}^{l,N}-X_{t}^{k,N})\right) V_{t,q}^{l,N}dt \\
=&\dfrac{2}{N}\sum_{k=1}^{N}\left(\dfrac{1}{N}\sum_{l=1}^{N}\nabla^{q}\phi_{N}(X_{t}^{k,N}-X_{t}^{l,N})\right) V_{t,q}^{k,N}dt\\
=&\dfrac{2}{N}\sum_{k=1}^{N}\left(\dfrac{1}{N}\sum_{l=1}^{N}\nabla\phi_{N}(X_{t}^{k,N}-X_{t}^{l,N})\right)\cdot V_{t}^{k,N}dt\\
=&\dfrac{2}{N}\sum_{k=1}^{N}V_{N}^{k}\cdot\nabla\left(S_{t}^{N}\ast\phi_{N}\right)(X_{t}^{k,N})dt
\end{align*}
By Leibniz rule  we obtain 
\begin{align*}
d\left(-\dfrac{2}{N}\sum_{k=1}^{N}\left(\varrho(.,t)\ast\phi_{N}^{r}\right)(X_{t}^{k,N})\right)
=&-\dfrac{2}{N}\sum_{k=1}^{N}d\left(\left(\varrho(.,t)\ast\phi_{N}^{r}\right)(X_{t}^{k,N})\right)\\
=&-\dfrac{2}{N}\sum_{k=1}^{N}d\left(\int\varrho(x,t)\phi_{N}^{r}(x-X_{t}^{k,N})dx\right)\\
=&-\dfrac{2}{N}\sum_{k=1}^{N}V_{t}^{k,N}\cdot\nabla\left(\varrho(.,t)\ast\phi_{N}^{r}\right)(X_{t}^{k,N})dt\\
&-\dfrac{2}{N}\sum_{k=1}^{N}\left( d\varrho(.,t)\ast\phi_{N}^{r}\right)(X_{t}^{k,N})dt\\
=&-\dfrac{2}{N}\sum_{k=1}^{N}V_{t}^{k,N}\cdot\nabla\left(\varrho(.,t)\ast\phi_{N}^{r}\right)(X_{t}^{k,N})dt\\
&+\dfrac{2}{N}\sum_{k=1}^{N}\left(\operatorname{div}_{x}(\varrho(\cdot,t)\upsilon(\cdot,t))\ast\phi_{N}^{r}\right)(X_{t}^{k,N})dt.
\end{align*}
and 
\begin{align*}
d\left(\Vert\varrho(.,t)\Vert_{2}^{2}\right)=&-2\langle\varrho(\cdot,t),\operatorname{div}_{x}(\varrho(\cdot,t)\upsilon(\cdot,t))\rangle dt.
\end{align*}
Adding all terms  we have
\begin{align*}
&d(Q_{t}^{N})\\
=&-\dfrac{2}{N}\sum_{k=1}^{N}V_{t,q}^{k,N}\nabla^{q}\left(S_{t}^{N}\ast\phi_{N}\right)(X_{t}^{k,N})dt+\dfrac{2}{N}\sum_{k=1}^{N}V_{t,q}^{k,N}  \sigma_{q}(X_{t}^{k,N}) V_{t,q}^{k,N}  \circ d B_{t}^{q}\\
&+\dfrac{2}{N}\sum_{k=1}^{N}\upsilon(X_{t}^{k,N},t)\cdot\nabla\left(S_{t}^{N}\ast\phi_{N}\right)(X_{t}^{k,N})dt\\
&+\dfrac{2}{N}\sum_{k=1}^{N}V_{t,q}^{k,N}\left(\upsilon(X_{t}^{k,N},t)\cdot\nabla\upsilon_{q}(X_{t}^{k,N},t)+\nabla^{q}\varrho(X_{t}^{k,N},t)\right)dt\\
&-\dfrac{2}{N}\sum_{k=1}^{N}V_{t,q}^{k,N}\nabla\upsilon_{q}(X_{t}^{k,N},t)\cdot V_{t}^{k,N}+\dfrac{2}{N}\sum_{k=1}^{N}\upsilon_{q}(X_{t}^{k,N},t) \sigma_{q}(X_{t}^{k,N}) \upsilon_{q}(X_{t}^{k,N},t) \circ dB_{t}^{q}\\
&- \dfrac{2}{N}\sum_{k=1}^{N} V_{t,q}^{k,N}  \sigma_{q}(X_{t}^{k,N}) \upsilon_{q}(X_{t}^{k,N},t) \circ dB_{t}^{q}\\
&-\dfrac{2}{N}\sum_{k=1}^{N}\upsilon_{q}(X_{t}^{k,N},t)\left(\upsilon(X_{t}^{k,N},t)\cdot\nabla\upsilon_{q}(X_{t}^{k,N},t)+\nabla^{q}\varrho(X_{t}^{k,N},t)\right)dt\\
&- \dfrac{2}{N}\sum_{k=1}^{N}\upsilon_{q}(X_{t}^{k,N},t) \sigma_{q}(X_{t}^{k,N}) V_{t,q}^{k,N}   \circ dB_{t}^{q}\\
& +\dfrac{2}{N}\sum_{k=1}^{N}\upsilon_{q}(X_{t}^{k,N},t)\nabla\upsilon_{q}(X_{t}^{k,N},t)\cdot V_{t}^{k,N}dt+\dfrac{2}{N}\sum_{k=1}^{N}V_{t,q}^{k,N}\nabla^{q}\left(S_{t}^{N}\ast\phi_{N}\right)(X_{t}^{k,N})dt\\
&-\dfrac{2}{N}\sum_{k=1}^{N}V_{t}^{k,N}\cdot\nabla\left(\varrho(.,t)\ast\phi_{N}^{r}\right)(X_{t}^{k,N})dt\\
&+\dfrac{2}{N}\sum_{k=1}^{N}\left(\operatorname{div}_{x}(\varrho(\cdot,t)\upsilon(\cdot,t))\ast\phi_{N}^{r}\right)(X_{t}^{k,N})dt-2\langle\varrho(\cdot,t),\operatorname{div}_{x}(\varrho(\cdot,t)\upsilon(\cdot,t))\rangle dt
\end{align*}
 Doing obvious cancellations we deduce
\begin{align*}
&d(Q_{t}^{N})\\
=&\dfrac{2}{N}\sum_{k=1}^{N}\upsilon(X_{t}^{k,N},t)\cdot\nabla\left(S_{t}^{N}\ast\phi_{N}\right)(X_{t}^{k,N})dt+\dfrac{2}{N}\sum_{k=1}^{N}V_{t,q}^{k,N}\upsilon(X_{t}^{k,N},t)\cdot\nabla\upsilon_{q}(X_{t}^{k,N},t)dt\\
&+\dfrac{2}{N}\sum_{k=1}^{N}V_{t}^{k,N}\cdot\nabla\varrho(X_{t}^{k,N},t)dt-\dfrac{2}{N}\sum_{k=1}^{N}V_{t,q}^{k,N}\nabla\upsilon_{q}(X_{t}^{k,N},t)\cdot V_{t}^{k,N} dt\\
&-\dfrac{2}{N}\sum_{k=1}^{N}\upsilon_{q}(X_{t}^{k,N},t)\upsilon(X_{t}^{k,N},t)\cdot\nabla\upsilon_{q}(X_{t}^{k,N},t)dt-\dfrac{2}{N}\sum_{k=1}^{N}\upsilon(X_{t}^{k,N},t)\cdot\nabla\varrho(X_{t}^{k,N},t)dt\\
\end{align*}

\begin{align*}
&+\dfrac{2}{N}\sum_{k=1}^{N}\upsilon_{q}(X_{t}^{k,N},t)\nabla\upsilon_{q}(X_{t}^{k,N},t)\cdot V_{t}^{k,N}dt-\dfrac{2}{N}\sum_{k=1}^{N}V_{t}^{k,N}\cdot\nabla\left(\varrho(.,t)\ast\phi_{N}^{r}\right)(X_{t}^{k,N})dt\\
&+\dfrac{2}{N}\sum_{k=1}^{N}\left(\operatorname{div}_{x}(\varrho(\cdot,t)\upsilon(\cdot,t))\ast\phi_{N}^{r}\right)(X_{t}^{k,N})dt-2\langle\varrho(\cdot,t),\operatorname{div}_{x}(\varrho(\cdot,t)\upsilon(\cdot,t))\rangle dt\\
&+\dfrac{2}{N}\sum_{k=1}^{N}  V_{t,q}^{k,N}  \sigma_{q}(X_{t}^{k,N}) V_{t,q}^{k,N}  \circ d B_{t}^{q}+\dfrac{2}{N}\sum_{k=1}^{N} \upsilon_{q}(X_{t}^{k,N},t) \sigma_{q}(X_{t}^{k,N}) \upsilon_{q}(X_{t}^{k,N},t) \circ dB_{t}^{q}\\
& - \dfrac{4}{N}\sum_{k=1}^{N}  \upsilon_{q}(X_{t}^{k,N},t) \sigma_{q}(X_{t}^{k,N}) V_{t,q}^{k,N} \circ dB_{t}^{q}.
\end{align*}
From proposition \ref{Ito-correction} we have 
\begin{align}\label{stoch1}
&\dfrac{2}{N}\sum_{k=1}^{N} V_{t,q}^{k,N}  \sigma_{q}(X_{t}^{k,N}) V_{t,q}^{k,N} \circ d B_{t}^{q}\nonumber\\
=&\dfrac{2}{N}\sum_{k=1}^{N} V_{t,q}^{k,N}  \sigma_{q}(X_{t}^{k,N}) V_{t,q}^{k,N}  d B_{t}^{q}+ \dfrac{2}{N}\sum_{k=1}^{N} |V_{t,q}^{k,N}|^{2}  |\sigma_{q}(X_{t}^{k,N})|^{2}  dt,
\end{align}
\begin{align}\label{stoch2}
&\dfrac{2}{N}\sum_{k=1}^{N}  \upsilon_{q}(X_{t}^{k,N},t) \sigma_{q}(X_{t}^{k,N}) \upsilon_{q}(X_{t}^{k,N},t)  \circ dB_{t}^{q}\nonumber\\
=&\dfrac{2}{N}\sum_{k=1}^{N} \upsilon_{q}(X_{t}^{k,N},t) \sigma_{q}(X_{t}^{k,N}) \upsilon_{q}(X_{t}^{k,N},t) dB_{t}^{q}+ \dfrac{2}{N}\sum_{k=1}^{N} |\upsilon_{q}(X_{t}^{k,N},t)|^{2} |\sigma_{q}(X_{t}^{k,N})|^{2} dt,
\end{align}
and
\begin{align}\label{stoch3}
&\dfrac{2}{N}\sum_{k=1}^{N} \upsilon_{q}(X_{t}^{k,N},t) \sigma_{q}(X_{t}^{k,N}) V_{t,q}^{k,N}  \circ dB_{t}^{q}\nonumber\\
=& \dfrac{2}{N}\sum_{k=1}^{N} \upsilon_{q}(X_{t}^{k,N},t) \sigma_{q}(X_{t}^{k,N}) V_{t,q}^{k,N} dB_{t}^{q}+  \dfrac{2}{N}\sum_{k=1}^{N} \upsilon_{q}(X_{t}^{k,N},t) |\sigma_{q}(X_{t}^{k,N})|^{2} V_{t,q}^{k,N} dt.
\end{align}
From (\ref{stoch1}), (\ref{stoch2}), (\ref{stoch3}) we have 
\begin{align*}
  &d\left(Q_{t}^{N}\right)\\
 =&\dfrac{2}{N}\sum_{k=1}^{N}\upsilon(X_{t}^{k,N},t)\cdot\nabla\left(S_{t}^{N}\ast\phi_{N}\right)(X_{t}^{k,N})dt\\
 &+\dfrac{2}{N}\sum_{k=1}^{N}V_{t,q}^{k,N}\upsilon(X_{t}^{k,N},t)\cdot\nabla\upsilon_{q}(X_{t}^{k,N},t)dt+\dfrac{2}{N}\sum_{k=1}^{N}V_{t}^{k,N}\cdot\nabla\varrho(X_{t}^{k,N},t)dt\\   
\end{align*}
\begin{align*}
&-\dfrac{2}{N}\sum_{k=1}^{N}V_{t,q}^{k,N}\nabla\upsilon_{q}(X_{t}^{k,N},t)\cdot V_{t}^{k,N} dt- \dfrac{2}{N}\sum_{k=1}^{N}\upsilon_{q}(X_{t}^{k,N},t)\upsilon(X_{t}^{k,N},t)\cdot\nabla\upsilon_{q}(X_{t}^{k,N},t)dt\\
&-\dfrac{2}{N}\sum_{k=1}^{N}\upsilon(X_{t}^{k,N},t)\cdot\nabla\varrho(X_{t}^{k,N},t)dt+ \dfrac{2}{N}\sum_{k=1}^{N}\upsilon_{q}(X_{t}^{k,N},t)\nabla\upsilon_{q}(X_{t}^{k,N},t)\cdot V_{t}^{k,N}dt\\
&-\dfrac{2}{N}\sum_{k=1}^{N}V_{t}^{k,N}\cdot\nabla\left(\varrho(.,t)\ast\phi_{N}^{r}\right)(X_{t}^{k,N})dt+\dfrac{2}{N}\sum_{k=1}^{N}\left(\operatorname{div}_{x}(\varrho(\cdot,t)\upsilon(\cdot,t))\ast\phi_{N}^{r}\right)(X_{t}^{k,N})dt\\
&-2\langle\varrho(\cdot,t),\operatorname{div}_{x}(\varrho(\cdot,t)\upsilon(\cdot,t))\rangle dt+    \dfrac{2}{N}\sum_{k=1}^{N} |V_{t,q}^{k,N}|^{2}  |\sigma_{q}(X_{t}^{k,N})|^{2}  dt\\
&+ \dfrac{2}{N}\sum_{k=1}^{N} |\upsilon_{q}(X_{t}^{k,N},t)|^{2}  |\sigma_{q}(X_{t}^{k,N})|^{2}  dt- \dfrac{2}{N}\sum_{k=1}^{N} \upsilon_{q}(X_{t}^{k,N},t) |\sigma_{q}(X_{t}^{k,N})|^{2} V_{t,q}^{k,N} dt \\
&-\dfrac{2}{N}\sum_{k=1}^{N}  \upsilon_{q}(X_{t}^{k,N},t) |\sigma_{q}(X_{t}^{k,N}) |^{2} V_{t,q}^{k,N}dt+  \dfrac{2}{N}\sum_{k=1}^{N} \sigma_{q}(X_{t}^{k,N})  |V_{t,q}^{k,N}-\upsilon_{q}(X_{t}^{k,N},t) |^{2}  dB_{t}^{q} \\
=&\sum_{j=1}^{14}   D_{N}(j,t)dt  + D_{N}(15,t) dB_{t}^{q}.
\end{align*}
From the definition of the empirical measure and convolution,  and integration by part we deduce 
\begin{align*}
&D_{N}(9,t)=2\left\langle S_{t}^{N},\left(\left(\operatorname{div}_{x}\left(\varrho(.,t)\upsilon(.,t)\right)\right)\ast\phi_{N}^{r}\right)(.)\right\rangle\\
=&2\left\langle (S_{t}^{N}\ast\phi_{N}^{r})(.),\operatorname{div}_{x}\left(\varrho(.,t)\upsilon(.,t)\right)\right\rangle=2\int_{\R^{d}}(S_{t}^{N}\ast\phi_{N}^{r})(x)\operatorname{div}_{x}\left(\varrho(x,t)\upsilon(x,t)\right)\thinspace dx\\
=&2\int_{\R^{d}}\operatorname{div}_{x}\left((S_{t}^{N}\ast\phi_{N}^{r})(x)\varrho(x,t)\upsilon(x,t)\right)\thinspace dx-2\int_{\R^{d}}\nabla(S_{t}^{N}\ast\phi_{N}^{r})(x)\cdot\left(\varrho(x,t)\upsilon(x,t)\right)\thinspace dx\\
=&-2\int_{\R^{d}}\nabla(S_{t}^{N}\ast\phi_{N}^{r})(x)\cdot\left(\varrho(x,t)\upsilon(x,t)\right)\thinspace dx\\
=&-2\left\langle\nabla(S_{t}^{N}\ast\phi_{N}^{r})(.),\varrho(.,t)\upsilon(.,t)\right\rangle.
\end{align*}
From the definition of the empirical measure and integration by parts we have
\begin{align}\label{AN1}
A_{N}(1,t):=&D_{N}(1,t)+D_{N}(6,t)+D_{N}(9,t)+D_{N}(10,t)\nonumber\\
=&\dfrac{2}{N}\sum_{k=1}^{N} \upsilon(X_{t}^{k,N},t)\cdot\nabla(S_{t}^{N}\ast\phi_{N})(X_{t}^{k,N})-\dfrac{2}{N}\sum_{k=1}^{N}\upsilon(X_{t}^{k,N},t)\cdot\nabla\varrho(,t)\nonumber\\
&-2\left\langle\nabla(S_{t}^{N}\ast\phi_{N}^{r})(.),\varrho(.,t)\upsilon(.,t)\right\rangle-2\left\langle\varrho(.,t),\operatorname{div}_{x}(\varrho(.,t)\upsilon(.,t))\right\rangle\nonumber\\
=&2\left\langle S_{t}^{N},\nabla(S_{t}^{N}\ast\phi_{N})(.)\cdot\upsilon(,t)\right\rangle-2\left\langle S_{t}^{N},\nabla\varrho(,t)\cdot\upsilon(.,t)\right\rangle\nonumber\\
&-2\left\langle\varrho(.,t),\nabla(S_{t}^{N}\ast\phi_{N}^{r})(.)\cdot\upsilon(.,t)\right\rangle+2\left\langle\varrho(.,t),\nabla\varrho(.,t)\cdot\upsilon(.,t)\right\rangle.
\end{align}
 We have 
\begin{align}\label{AN2}
A_{N}(2,t):=&D_{N}(2,t)+D_{N}(4,t)+D_{N}(7,t)+D_{N}(5,t)\nonumber\\
=&\dfrac{2}{N}\sum_{k=1}^{N}V_{t,q}^{k,N}\upsilon(X_{t}^{k,N},t)\cdot\nabla\upsilon_{q}(X_{t}^{k,N},t)-\dfrac{2}{N}
\sum_{k=1}^{N}V_{t,q}^{k,N}\nabla\upsilon_{q}(X_{t}^{k,N},t) \cdot V_{t}^{k,N}\nonumber\\
&+\dfrac{2}{N}\sum_{k=1}^{N}\upsilon_{q}(X_{t}^{k,N},t)\nabla\upsilon_{q}(X_{t}^{k,N},t) \cdot V_{t}^{k,N}\nonumber\\
&-\dfrac{2}{N}
\sum_{k=1}^{N}\upsilon_{q}(X_{t}^{k,N},t)\upsilon(X_{t}^{k,N},t)\cdot\nabla\upsilon_{q}(X_{t}^{k,N},t).
\end{align}
Conveniently grouping and applying Young inequality we get
\begin{align}\label{cotaAN2}
|A_{N}(2,t)|=&\dfrac{2}{N}\bigg|\sum_{k=1}^{N}V_{t,q}^{k,N}\upsilon_{q^{\prime}}(X_{t}^{k,N},t)\nabla^{q^{\prime}}\upsilon_{q}(X_{t}^{k,N},t)-V_{t,q}^{k,N}\nabla^{q^{\prime}}\upsilon_{q}(X_{t}^{k,N},t) 
V_{t,q^{\prime}}^{k,N}\nonumber\\
+&\upsilon_{q}(X_{t}^{k,N},t)\nabla^{q^{\prime}}\upsilon_{q}(X_{t}^{k,N},t) V_{t,q^{\prime}}^{k,N}-\upsilon_{q}(X_{t}^{k,N},t)\upsilon_{q^{\prime}}(X_{t}^{k,N},t)\nabla^{q^{\prime}}\upsilon_{q}
(X_{t}^{k,N},t)\bigg|\nonumber\\
=&\frac{2}{N}\bigg|\sum_{k=1}^{N}-\nabla^{q^{\prime}}\upsilon_{q}(X_{t}^{k,N},t)\left(V_{t,q}^{k,N}-\upsilon_{q}(X_{t}^{k,N},t)\right)\left(V_{t,q^{\prime}}^{k,N}-\upsilon_{q^{\prime}}(X_{t}^{k,N},t)\right)\bigg|\nonumber\\
\leq& C\frac{1}{N}\sum_{q,q^{\prime}=1}^{d}\sum_{k=1}^{N}\Big|V_{t,q}^{k,N}-\upsilon_{q}(X_{t}^{k,N},t)\Big|\thinspace\Big|V_{t,q^{\prime}}^{k,N}-\upsilon_{q^{\prime}}(X_{t}^{k,N},t)\Big|\nonumber\\
\leq& C\frac{1}{N}\sum_{k=1}^{N}\Big|V_{t}^{k,N}-\upsilon(X_{t}^{k,N},t)\Big|^{2}.
\end{align}
We have
\begin{align}\label{AN3}
A_{N}(3,t):=& D_{N}(3,t)+D_{N}(8,t)\nonumber\\
=&\dfrac{2}{N}\sum_{k=1}^{N}V_{t}^{k,N}\cdot\nabla\varrho(X_{t}^{k,N},t)-\dfrac{2}{N}\sum_{k=1}^{N}V_{t}^{k,N}\cdot\nabla\left(\varrho(.,t)\ast\phi_{N}^{r}\right)(X_{t}^{k,N})\nonumber\\
=&\dfrac{2}{N}\sum_{k=1}^{N}V_{t}^{k,N}\cdot\left(\nabla\varrho(X_{t}^{k,N},t)-\nabla\left(\varrho(.,t)\ast\phi_{N}^{r}\right)(X_{t}^{k,N})\right).
\end{align}
By Holder inequality   we deduce 
\begin{align}\label{cotaAN3}
|A_{N}(3,t)|=&\frac{2}{N}\left|\sum_{k=1}^{N}V_{t}^{k,N}\cdot\left(\nabla\varrho(X_{t}^{k,N},t)-\nabla\left(\varrho(.,t)\ast\phi_{N}^{r}\right)(X_{t}^{k,N})\right)\right|\nonumber\\
\leq&\frac{2}{N}\sum_{k=1}^{N}\left|V_{t}^{k,N}\cdot\left(\nabla\varrho(X_{t}^{k,N},t)-\nabla\left(\varrho(.,t)\ast\phi_{N}^{r}\right)(X_{t}^{k,N})\right)\right|\nonumber\\
\leq&\frac{2}{N}\sum_{k=1}^{N}\left|V_{t}^{k,N}\right|\left|\left(\nabla\varrho(X_{t}^{k,N},t)-\nabla\left(\varrho(.,t)\ast\phi_{N}^{r}\right)(X_{t}^{k,N})\right)\right|\nonumber\\
\leq&\frac{2}{N}\sum_{k=1}^{N}\left|V_{t}^{k,N}\right|\Vert\nabla\varrho(.,t)-\nabla\varrho(.,t)\ast\phi_{N}^{r}(.)\Vert_{\infty}\nonumber\\
=&2\Vert\nabla\varrho(.,t)-\nabla\varrho(.,t)\ast\phi_{N}^{r}(.)\Vert_{\infty}\sum_{k=1}^{N}\frac{\left|V_{t}^{k,N}\right|}{N}\nonumber\\
\leq&2\Vert\nabla\varrho(.,t)-\nabla\varrho(.,t)\ast\phi_{N}^{r}(.)\Vert_{\infty}\left(\sum_{k=1}^{N}\frac{\left|V_{t}^{k,N}\right|^{2}}{N}\right)^{1/2}\left(\sum_{k=1}^{N}\frac{1}{N}\right)^{1/2}\nonumber\\
=&2\Vert\nabla\varrho(.,t)-\nabla\varrho(.,t)\ast\phi_{N}^{r}(.)\Vert_{\infty}\left(\frac{1}{N}\sum_{k=1}^{N}\left|V_{t}^{k,N}\right|^{2}\right)^{1/2}\nonumber\\
\leq&2\left(\frac{1}{N}\sum_{k=1}^{N}\left|V_{t}^{k,N}\right|^{2}\right)^{1/2}\sum_{i=1}^{d}\Vert\nabla^{i}\varrho(.,t)-\nabla^{i}\varrho(.,t)\ast\phi_{N}^{r}(.)\Vert_{\infty}\nonumber\\
\leq&2\left(\frac{1}{N}\sum_{k=1}^{N}\left|V_{t}^{k,N}\right|^{2}\right)^{1/2}N^{-\beta/d}\sum_{i=1}^{d}C_{0}\Vert\nabla(\nabla^{i}\varrho(.,t))\Vert_{\infty}\nonumber\\
\leq&N^{-\beta/d}C\left(\frac{1}{N}\sum_{k=1}^{N}\left|\upsilon(X_{t}^{k,N},t)\right|^{2}+\frac{1}{N}\sum_{k=1}^{N}\left|V_{t}^{k,N}-\upsilon(X_{t}^{k,N},t)\right|^{2}\right)^{1/2}\nonumber\\
\leq&N^{-\beta/d}C\left(\Vert\upsilon(.,t)\Vert_{\infty}^{2}+\frac{1}{N}\sum_{k=1}^{N}\left|V_{t}^{k,N}-\upsilon(X_{t}^{k,N},t)\right|^{2}\right)^{1/2}\nonumber\\
\leq&N^{-\beta/d}C\max\{\Vert\upsilon(.,t)\Vert_{\infty},1\}\left(1+\frac{1}{N}\sum_{k=1}^{N}\left|V_{t}^{k,N}-\upsilon(X_{t}^{k,N},t)\right|^{2}\right)^{1/2}\nonumber\\
\leq&N^{-\beta/d}C\max\{\Vert\upsilon(.,t)\Vert_{\infty},1\}\left(1+\frac{1}{N}\sum_{k=1}^{N}\left|V_{t}^{k,N}-\upsilon(X_{t}^{k,N},t)\right|^{2}\right)\nonumber\\
\leq&C\left(N^{-\beta/d}+\frac{1}{N}\sum_{k=1}^{N}\left|V_{t}^{k,N}-\upsilon(X_{t}^{k,N},t)\right|^{2}\right).
\end{align}
We observe $A_{N}(1,t)$  that 
\begin{align}\label{ANRN}
A_{N}(1,t)=&2\langle S_{t}^{N}-\varrho(.,t),\nabla\left(\left(S_{t}^{N}-\varrho(.,t)\right)\ast\phi_{N}\right)(.)\cdot\upsilon(.,t)\rangle+R_{N}(t),\nonumber\\
\end{align}
where 
\begin{align*}
   R_{N}(t)=&2\langle S_{t}^{N}-\varrho(.,t),\nabla\left(\varrho(.,t)\ast\phi_{N}^{r}-\varrho(.,t)\right)(.)\cdot\upsilon(.,t) \rangle\\
   &+\quad2\langle \varrho(.,t),\nabla\left(S_{t}^{N}\ast\phi_{N}-S_{t}^{N}\ast\phi_{N}^{r}\right)(.)\cdot\upsilon(.,t)\rangle.
\end{align*}
By integration by parts and simple calculations we have 
\begin{align}\label{RN}
R_{N}(t)=&2\langle S_{t}^{N}-\varrho(.,t),\nabla\left(\varrho(.,t)\ast\phi_{N}^{r}-\varrho(.,t)\right)(.)\cdot\upsilon(.,t)\rangle\nonumber\\
 &-2\langle \operatorname{div}_{x}(\varrho(\cdot,t)\upsilon(.,t)),\left(S_{t}^{N}\ast\left(\phi_{N}-\phi_{N}^{r}\right)\right)(\cdot)\rangle\nonumber\\
=&2\langle S_{t}^{N}-\varrho(.,t),\nabla\left(\varrho(.,t)\ast\phi_{N}^{r}-\varrho(.,t)\right)(.)\cdot\upsilon(.,t)\rangle\nonumber\\
 &-2\langle S_{t}^{N},\left(\operatorname{div}_{x}(\varrho(\cdot,t)\upsilon(.,t))\ast\left(\phi_{N}-\phi_{N}^{r}\right)\right)(\cdot)\rangle\nonumber\\
=&2\langle S_{t}^{N}-\varrho(.,t),\nabla\left(\varrho(.,t)\ast\phi_{N}^{r}-\varrho(.,t)\right)(.)\cdot\upsilon(.,t)\rangle\nonumber\\
 &+2\langle S_{t}^{N},\left(\operatorname{div}_{x}(\varrho(\cdot,t)\upsilon(.,t))\ast\left(\phi_{N}^{r}-\phi_{N}\right)\right)(\cdot)\rangle.
\end{align}
Using that  $S_{t}^{N},\varrho$  are probability densities  and \eqref{cotafNbeta} we have  
\begin{align}\label{cotaRN}
|R_{N}(t)|\leq& 2|\langle S_{t}^{N}-\varrho(.,t),\nabla\left(\varrho(.,t)\ast\phi_{N}^{r}-\varrho(.,t)\right)(.)\cdot\upsilon(.,t)\rangle|\nonumber\\
&\qquad+2\langle S_{t}^{N},|\left(\operatorname{div}_{x}(\varrho(\cdot,t)\upsilon(.,t))\ast\left(\phi_{N}^{r}-\phi_{N}\right)\right)(\cdot)|\rangle\nonumber\\
\leq&2\langle S_{t}^{N}+\varrho(.,t),1\rangle N^{-\beta/d}\sum_{i=1}^{d}C_{0}\Vert\nabla(\nabla^{i}\varrho(.,t))\Vert_{\infty}\Vert\upsilon(.,t)\Vert_{\infty}\nonumber\\
&\qquad+2N^{-\beta/d}C_{0}\Vert\nabla\left(\operatorname{div}_{x}(\varrho(\cdot,t)\upsilon(.,t))\ast\phi_{N}^{r}\right)\Vert_{\infty}\nonumber\\
=&2\lambda N^{-\beta/d}\Vert\upsilon(.,t)\Vert_{\infty}+2N^{-\beta/d}C_{0}\Vert\left[\nabla\left(\operatorname{div}_{x}(\varrho(\cdot,t)\upsilon(.,t))\right)\right]\ast\phi_{N}^{r}\Vert_{\infty}\nonumber\\
\leq&2\lambda N^{-\beta/d}\Vert\upsilon(.,t)\Vert_{\infty}+2N^{-\beta/d}C_{0}d\Vert\nabla\left(\operatorname{div}_{x}(\varrho(\cdot,t)\upsilon(.,t))\right)\Vert_{\infty}\Vert\phi_{N}^{r}\Vert_{L^{1}}\nonumber\\
=&2\lambda N^{-\beta/d}\Vert\upsilon(.,t)\Vert_{\infty}+2N^{-\beta/d}C_{0}d\Vert\nabla\left(\operatorname{div}_{x}(\varrho(\cdot,t)\upsilon(.,t))\right)\Vert_{\infty}\Vert\phi_{1}^{r}\Vert_{L^{1}}\nonumber\\
\leq&CN^{-\beta/d}.
\end{align}
We postpone the proof of 
\begin{align}\label{restoA1}
|A_{N}(1,t)-R_{N}(t)|=&2|\langle S_{t}^{N}-\varrho(.,t),\nabla\left(\left(S_{t}^{N}-\varrho(.,t)\right)\ast\phi_{N}\right)(.)\cdot\upsilon(.,t)\rangle|\nonumber\\
\leq& C_{1}\left(\Vert S_{t}^{N}\ast\phi_{N}^{r}-\varrho(.,t)\Vert_{0}^{2}+N^{-\beta/d}\right).
\end{align}
We denoted 
\begin{align*}
A_{N}(4,t)= D_{N}(11,t)+  D_{N}(12,t) +  D_{N}(13,t) +  D_{N}(14,t).
\end{align*}

The we have 

\begin{eqnarray}\label{stocfinal}
 A_{N}(4,t) & \leq &   C_{\sigma}  \dfrac{1}{N}\sum_{k=1}^{N} | V_{t}^{k,N}-\upsilon_{q}(X_{t}^{k,N},t)|^{2}.  
\end{eqnarray}

From  \eqref{cotaAN2}, \eqref{cotaAN3}, \eqref{cotaRN}, \eqref{restoA1},  
(\ref{stocfinal}) and $\EE( D_{N}(15,t) dB_{t}^{q})=0$ we conclude that there exist   $C_{m,T}>0$ such that 
\begin{eqnarray*}
\EE Q_{t\wedge \tau_{m}}^{N}  &\leq&  \EE Q_{0}^{N} + C\int_{0}^{t}  \left(\EE  Q_{s\wedge \tau_{m}}^{N}+N^{-\beta/d}\right)ds\qquad0\leq t\leq T.
\end{eqnarray*}
Then by Gronwall lemma we conclude 
\begin{eqnarray}
 \EE  Q_{t\wedge \tau_{m}}^{N}\leq C_{m,T}(\EE Q_{0}^{N}+N^{-\beta/d})\qquad\forall N\in\N. 
\end{eqnarray}

We will show  \eqref{restoA1}. We denoted 
\begin{equation}\label{defYN}
    Y_{N}^{t}(dx):=S_{t}^{N}(dx)-\varrho(.,t)dx.
\end{equation}
By Taylor expansion we have 
\begin{align*}
&\left\langle Y_{N},\nabla(Y_{N}\ast\phi_{N})(.)\cdot\upsilon\right\rangle\nonumber=\left\langle Y_{N},((Y_{N}\ast\phi_{N}^{r})\ast\nabla\phi_{N}^{r})(.)\cdot\upsilon\right\rangle\nonumber\\
=&\left\langle Y_{N},\int_{\R^{d}}\nabla^{q}\phi_{N}^{r}(z)\upsilon_{q}(.)(Y_{N}\ast\phi_{N}^{r})(\cdot-z)dz\right\rangle\nonumber\\
=&\Biggl\langle Y_{N},\int_{\R^{d}}\nabla^{q}\phi_{N}^{r}(z)\left\lbrace\sum_{0\leq|\alpha|\leq L}\dfrac{z^{\alpha}}{\alpha !} \partial^{\alpha}\upsilon_{q}(\cdot-z)+\sum_{|\alpha|=L+1}\dfrac{z^{\alpha}}{\alpha !} \partial^{\alpha}\upsilon_{q}(\cdot-z+\vartheta_{q}(z,.)z)\right\rbrace\Biggr.\\
 &\Biggl.\qquad\qquad\qquad\qquad\qquad\qquad \times(Y_{N}\ast\phi_{N}^{r})(\cdot-z)dz\Biggr\rangle, 
\end{align*}
with 
\begin{equation}\label{partesn/2 +1}
    L:=\left[\dfrac{d+2}{2}\right].
\end{equation}
We observe  that 
\begin{eqnarray*}
(Y_{N}\ast U_{N;\alpha}^{q})(x)&=&(S_{t}^{N}\ast U_{1;\alpha}^{q})(x)-(\varrho\ast U_{N;\alpha}^{q})(x)
\end{eqnarray*}
where
\begin{eqnarray*}
U_{N;\alpha}^{q}(x)&:=&N^{\beta(1+(1-|\alpha|)/d)}U_{1;\alpha}^{q}(N^{\beta/d}x),
\end{eqnarray*}
and 
\begin{align*}
U_{N;\alpha}^{q}(x)=&N^{\beta(1+(1-|\alpha|)/d)}(-1)^{1+|\alpha|}N^{\beta|\alpha|/d}\dfrac{x^{\alpha}}{\alpha !} \nabla_q\phi_{1}^{r}(N^{\beta/d}x)\\
=&(-1)^{1+|\alpha|}\dfrac{x^{\alpha}}{\alpha !} \nabla_q\phi_{N}^{r}(x).
\end{align*}
By simple calculation we  obtain
\begin{align*}
&\langle (S_{t}^{N}\ast U_{N;\alpha}^{q})(.),\partial^{\alpha}\upsilon_{q}(.)(Y_{N}\ast\phi_{N}^{r})(.)\rangle\\
=&\dfrac{1}{N}\sum_{k=1}^{N}\int_{\R^{d}}U_{N;\alpha}^{q}(w-X_{t}^{k,N})\partial^{\alpha}\upsilon_{q}(w)(Y_{N}\ast\phi_{N}^{r})(w)dw\\
=&\dfrac{(-1)^{1+|\alpha|}}{N}\sum_{k=1}^{N}\int_{\R^{d}}U_{N;\alpha}^{q}(X_{t}^{k,N}-w)\partial^{\alpha}\upsilon_{q}(w)(Y_{N}\ast\phi_{N}^{r})(w)dw\\
=&\left\langle S_{t}^{N},(-1)^{1+|\alpha|}\int_{\R^{d}}U_{N;\alpha}^{q}(\cdot-w)\partial^{\alpha}\upsilon_{q}(w)(Y_{N}\ast\phi_{N}^{r})(w)dw\right\rangle\\
=&\left\langle S_{t}^{N},(-1)^{1+|\alpha|}\int_{\R^{d}}U_{N;\alpha}^{q}(z)\partial^{\alpha}\upsilon_{q}(\cdot-z)(Y_{N}\ast\phi_{N}^{r})(\cdot-z)dz\right\rangle\\
=&\left\langle S_{t}^{N},\int_{\R^{d}}\nabla_q\phi_{N}^{r}(z)\dfrac{z^{\alpha}}{\alpha !} \partial^{\alpha}\upsilon_{q}(\cdot-z)(Y_{N}\ast\phi_{N}^{r})(\cdot-z)dz\right\rangle
\end{align*}
and 
\begin{align*}
&\left\langle(\varrho\ast U_{N;\alpha}^{q})(.),\partial^{\alpha}\upsilon_{q}(.)(Y_{N}\ast\phi_{N}^{r})(.)\right\rangle\\
=&\int_{\R^{d}}(\varrho\ast U_{N;\alpha}^{q})(x)\partial^{\alpha}\upsilon_{q}(x)(Y_{N}\ast\phi_{N}^{r})(x)dx\\
=&\int_{\R^{d}}\partial^{\alpha}\upsilon_{q}(x)(Y_{N}\ast\phi_{N}^{r})(x)\left(\int_{\R^{d}}\varrho(x-z)U_{N;\alpha}^{q}(z)dz\right)dx\\
=&\int_{\R^{d}}U_{N;\alpha}^{q}(z)\left(\int_{\R^{d}}\varrho(x-z)\partial^{\alpha}\upsilon_{q}(x)(Y_{N}\ast\phi_{N}^{r})(x)dx\right)dz\\
=&\int_{\R^{d}}U_{N;\alpha}^{q}(z)\left(\int_{\R^{d}}\varrho(w)\partial^{\alpha}\upsilon_{q}(w+z)(Y_{N}\ast\phi_{N}^{r})(w+z)dw\right)dz\\
=&\int_{\R^{d}}\varrho(w)\left(\int_{\R^{d}}U_{N;\alpha}^{q}(z)\partial^{\alpha}\upsilon_{q}(w+z)(Y_{N}\ast\phi_{N}^{r})(w+z)dz\right)dw\\
=&\int_{\R^{d}}\varrho(w)\left(\int_{\R^{d}}U_{N;\alpha}^{q}(-z)\partial^{\alpha}\upsilon_{q}(w-z)(Y_{N}\ast\phi_{N}^{r})(w-z)dz\right)dw\\
=&\int_{\R^{d}}\varrho(w)\left(\int_{\R^{d}}\nabla_q\phi_{N}^{r}(z)\dfrac{z^{\alpha}}{\alpha !} \partial^{\alpha}\upsilon_{q}(w-z)(Y_{N}\ast\phi_{N}^{r})(w-z)dz\right)dw.
\end{align*}
The we have 
\begin{align}\label{YNRNRN*}
\left\langle Y_{N},\nabla(Y_{N}\ast\phi_{N})(.)\cdot\upsilon\right\rangle\nonumber=\sum_{0\leq|\alpha|\leq L}R_{N;\alpha}+\sum_{|\alpha|=L+1}R_{N;\alpha}^{\ast}
\end{align}
where 
\begin{align*}
R_{N;\alpha}=&\langle (Y_{N}\ast U_{N;0}^{q})(.) ,\partial^{\alpha}\upsilon_{q}(.)(Y_{N}\ast\phi_{N}^{r})(.)\rangle\\
R_{N;\alpha}^{\ast}=&\left\langle Y_{N},\int_{\R^{d}}\nabla^{q}\phi_{N}^{r}(z)\dfrac{z^{\alpha}}{\alpha !} \partial^{\alpha}\upsilon_{q}(\cdot-z+\vartheta_{q}(z,.)z)(Y_{N}\ast\phi_{N}^{r})(\cdot-z)dz\right\rangle.
\end{align*}
We will estimated  $R_{N,\alpha}$ and  $R_{N,\alpha}^{*}$.

 For  $|\alpha|=0$ we have 
 \begin{eqnarray}\label{cotaR0}
|R_{N,\alpha}|&=&|\langle (Y_{N}\ast U_{N;\alpha}^{q})(.) ,\partial^{\alpha}\upsilon_{q}(.)(Y_{N}\ast\phi_{N}^{r})(.)\rangle|\nonumber\\
&=&|\langle (Y_{N}\ast \nabla^{q}\phi_{N}^{r})(.),\upsilon_{q}(.)(Y_{N}\ast\phi_{N}^{r})(.)\rangle|\nonumber\\
&=&|\langle(Y_{N}\ast \phi_{N}^{r})(.),\upsilon(.)\cdot\nabla(Y_{N}\ast\phi_{N}^{r})(.)\rangle|\nonumber\\
&=&\dfrac{1}{2}|\langle(Y_{N}\ast\phi_{N}^{r})^{2}(.),\operatorname{div}_{x}\upsilon(.)\rangle|\nonumber\\
&\leq&C\Vert Y_{N}\ast\phi_{N}^{r}\Vert_{L^{2}}^{2}.
\end{eqnarray}

 For $0<|\alpha|\leq L$, by Holder inequality we have 
\begin{align*}
|R_{N,\alpha}|=&|\langle (Y_{N}\ast U_{N;\alpha}^{q})(.) ,\partial^{\alpha}\upsilon_{q}(.)(Y_{N}\ast\phi_{N}^{r})(.)\rangle|\\
\leq&\Vert \partial^{\alpha}\upsilon(.,t)\Vert_{\infty}\sum_{q=1}^{d}\langle |Y_{N}\ast U_{N;\alpha}^{q}|,|Y_{N}\ast\phi_{N}^{r}|\rangle\\
\leq&\Vert \partial^{\alpha}\upsilon(.,t)\Vert_{\infty}\sum_{q=1}^{d}\Vert Y_{N}\ast U_{N;\alpha}^{q}\Vert_{L^{2}}\Vert Y_{N}\ast\phi_{N}^{r}\Vert_{L^{2}}.
\end{align*}

By Plancherel and simple calculation we obtain 
\begin{align*}
\Vert Y_{N}\ast U_{N;\alpha}^{q}\Vert_{L^{2}}^{2}=&(2\pi)^{d}\int_{\R^{d}}|\widetilde{Y_{N}(z)}|^{2}|\widetilde{U^{q}_{N;\alpha}}(z)|^{2}dz\\
=&(2\pi)^{d}N^{\beta(1+(1-|\alpha|)/d)}\int_{\R^{d}}|\widetilde{Y_{N}(z)}|^{2}|N^{-\beta}\widetilde{U_{1;\alpha}^{q}}(N^{-\beta/d}z)|^{2}dz\\
\leq &(2\pi)^{d}CN^{\beta(1+(1-|\alpha|)/d)}\int_{\R^{d}}|\widetilde{Y_{N}(z)}|^{2}|N^{-\beta}\widetilde{\phi_{1}^{r}}(N^{-\beta/d}z)|^{2}dz\\
=&(2\pi)^{d}CN^{\beta(1+(1-|\alpha|)/d)}\int_{\R^{d}}|\widetilde{Y_{N}(z)}|^{2}|N^{-\beta}\widetilde{\phi_{N}^{r}}(z)|^{2}dz\\
=&(2\pi)^{d}CN^{\beta((1-|\alpha|)/d)-1)}\int_{\R^{d}}|\widetilde{Y_{N}(z)}|^{2}|\widetilde{\phi_{N}^{r}}(z)|^{2}dz\\
\leq& C(2\pi)^{d}\int_{\R^{d}}|\widetilde{Y_{N}(z)}|^{2}|\widetilde{\phi_{N}^{r}}(z)|^{2}dz\\
=& C\Vert Y_{N}\ast\phi_{N}^{r}\Vert_{L^{2}}^{2}.
\end{align*}
Then we deduce 
\begin{eqnarray}\label{cotaRL}
|R_{N,\alpha}|&\leq&C\Vert Y_{N}\ast\phi_{N}^{r}\Vert_{L^{2}}^{2}
\end{eqnarray}
For  $|\alpha|=L+1$, we have 
\begin{align}
|R_{N;\alpha}^{\ast}|\leq&\sum_{q=1}^{d}\left\langle S_{t}^{N}+\varrho,\int_{\R^{d}}\left\vert\dfrac{z^{\alpha}}{\alpha !}\nabla^{q}\phi_{N}^{r}(z)\right\vert|\partial^{\alpha}\upsilon(\cdot-z+\vartheta_{q}(z,.)z)||(Y_{N}\ast\phi_{N}^{r})(\cdot-z)|dz\right\rangle\nonumber\\
\leq&\Vert \partial^{\alpha}\upsilon\Vert_{\infty}\sum_{q=1}^{d}\left\langle S_{t}^{N}+\varrho,\int_{\R^{d}}|U_{N;\alpha}^{q}(z)||(Y_{N}\ast\phi_{N}^{r})(\cdot-z)|dz\right\rangle\nonumber\\
=&\Vert \partial^{\alpha}\upsilon\Vert_{\infty}\sum_{q=1}^{d}\left\langle S_{t}^{N}+\varrho,\left(|U_{N;\alpha}^{q}|\ast|Y_{N}\ast\phi_{N}^{r}|\right)(.)\right\rangle\nonumber\\
=&\Vert \partial^{\alpha}\upsilon\Vert_{\infty}\sum_{q=1}^{d}\left\langle S_{t}^{N}\ast|U_{N;\alpha}^{q}|+\varrho\ast|U_{N;\alpha}^{q}|,|Y_{N}\ast\phi_{N}^{r}(.)|\right\rangle.
\end{align}
We observe that 
\begin{eqnarray}
\int_{R^{d}}\frac{1}{1+(N^{\beta/d}|y-x|)^{d+1}}dy=\int_{R^{d}}\frac{1}{1+|u|^{d+1}}N^{-\beta}du
\end{eqnarray}
\begin{eqnarray*}
   \left\Vert\left(\frac{1}{1+N^{\beta(d+1)/d}|\cdot-x|^{d+1}}\right)^{1/2}\right\Vert_{L^{2}}&=&N^{-\beta/2}\left\Vert\left(\frac{1}{1+|\cdot|^{d+1}}\right)^{1/2}\right\Vert_{L^{2}}.
\end{eqnarray*}
By hypothesis (\ref{cotauj})  and Holder inequality  we have 
\begin{align}
\langle S_{t}^{N}\ast|U_{N;\alpha}^{q}|,&|Y_{N}\ast\phi_{N}^{r}|\rangle=\int_{\R^{d}}|U_{N;\alpha}^{q}(y-x)|\int_{\R^{d}}|(Y_{N}\ast\phi_{N}^{r})(y)|dS_{t}^{N}dy\nonumber\\
=&N^{\beta(1-L/d)}\int_{\R^{d}}|U_{1;\alpha}^{q}(N^{\beta/d}(y-x))|\int_{\R^{d}}|(Y_{N}\ast\phi_{N}^{r})(y)|dS_{t}^{N}dy\nonumber\\
\leq&CN^{\beta(1-L/d)}\int_{\R^{d}}\left(\frac{1}{1+N^{\beta(d+1)/d}|y-x|^{d+1}}\right)^{1/2}\int_{\R^{d}}|(Y_{N}\ast\phi_{N}^{r})(y)|dS_{t}^{N}dy\nonumber\\
\leq&CN^{\beta(1-L/d)}\int_{\R^{d}}\int_{\R^{d}}\left(\frac{1}{1+N^{\beta(d+1)/d}|y-x|^{d+1}}\right)^{1/2}|(Y_{N}\ast\phi_{N}^{r})(y)|dydS_{t}^{N}\nonumber\\
\leq&CN^{\beta(1-L/d)}\int_{\R^{d}}\left\Vert\left(\frac{1}{1+N^{\beta(d+1)/d}|\cdot-x|^{d+1}}\right)^{1/2}\right\Vert_{L^{2}}\Vert Y_{N}\ast\phi_{N}^{r}\Vert_{L^{2}}dS_{t}^{N}\nonumber\\
\leq& CN^{\beta(1-L/d)}N^{-\beta/2}\Vert Y_{N}\ast\phi_{N}^{r}\Vert_{L^{2}}\nonumber\\
=&CN^{\beta(1/2-L/d)}\Vert Y_{N}\ast\phi_{N}^{r}\Vert_{L^{2}}
\end{align}
and 
\begin{eqnarray}\label{cotaRN*}
\langle\varrho\ast|U_{N;\alpha}^{q}|, |Y_{N}\ast\phi_{N}^{r}|\rangle\leq CN^{\beta(1/2-L/d)}\Vert Y_{N}\ast\phi_{N}^{r}\Vert_{L^{2}}.
\end{eqnarray}
Then 
\begin{eqnarray}
|R_{N;\alpha}^{\ast}|&\leq&CN^{\beta(1/2-L/d)}\Vert Y_{N}\ast\phi_{N}^{r}\Vert_{L^{2}}
\end{eqnarray}
From   \eqref{cotaR0}, \eqref{cotaRL},  and  \eqref{cotaRN*}, there exist  $C>0$ such that 
\begin{equation}
|\left\langle Y_{N},\nabla(Y_{N}\ast\phi_{N})(.)\cdot\upsilon\right\rangle|\leq C\left(\Vert Y_{N}\ast\phi_{N}^{r}\Vert_{L^{2}}^{2}+N^{\beta(1-2L/d)}\right).
\end{equation}
From \eqref{cotafNbeta} e \eqref{partesn/2 +1} we conclude 
\begin{align*}
|\left\langle Y_{N},\nabla(Y_{N}\ast\phi_{N})(.)\cdot\upsilon\right\rangle|\leq&C\left(\Vert Y_{N}\ast\phi_{N}^{r}\Vert_{L^{2}}^{2}+N^{-\beta/d}\right)\\
=&C\left(\Vert S_{t}^{N}\ast\phi_{N}^{r}-\varrho\ast\phi_{N}^{r}\Vert_{L^{2}}^{2}+N^{-\beta/d}\right)\\
\leq & 2C\left(\Vert S_{t}^{N}\ast\phi_{N}^{r}-\varrho\Vert_{L^{2}}^{2}+\Vert \varrho-\varrho\ast\phi_{N}^{r}\Vert_{L^{2}}^{2}+N^{-\beta/d}\right)\\
\leq & 2C\left(\Vert S_{t}^{N}\ast\phi_{N}^{r}-\varrho\Vert_{L^{2}}^{2}+\Vert \varrho-\varrho\ast\phi_{N}^{r}\Vert_{\infty}\Vert \varrho-\varrho\ast\phi_{N}^{r}\Vert_{L^{2}}\right.\\
&\qquad\qquad\left.+N^{-\beta/d}\right)\\
\leq & 2C\left(\Vert S_{t}^{N}\ast\phi_{N}^{r}-\varrho\Vert_{L^{2}}^{2}+C_{0}N^{-\beta/d}\Vert\nabla\varrho\Vert_{\infty}\Vert \varrho-\varrho\ast\phi_{N}^{r}\Vert_{L^{2}}\right.\\
&\qquad\qquad\left.+N^{-\beta/d}\right)\\
\leq&C\left(\Vert S_{t}^{N}\ast\phi_{N}^{r}-\varrho\Vert_{L^{2}}^{2}+N^{-\beta/d}\right).
\end{align*}
Finally we will show the convergence  of  $S_{t}^{N}$ e $V_{t}^{N}$ to  $\upsilon$ e $\varrho \upsilon$  respectively. We take $f\in   H^{\alpha}(  \mathbb{R}^{d})$, by  \eqref{cotafNbeta} we have 
\begin{align}\label{limXNvarr}
|\langle S_{t}^{N},f\rangle-\langle\varrho(.,t),f \rangle|=&|\langle S_{t}^{N},f-f\ast\phi_{N}^{r}+f\ast\phi_{N}^{r}\rangle-\langle\varrho(.,t),f \rangle|\nonumber\\
=&|\langle S_{t}^{N},f-f\ast\phi_{N}^{r}\rangle+\langle S_{t}^{N}\ast\phi_{N}^{r},f\rangle-\langle\varrho(.,t),f \rangle|\nonumber\\
=&|\langle S_{t}^{N},f-f\ast\phi_{N}^{r}\rangle+\langle S_{t}^{N}\ast\phi_{N}^{r}-\varrho(.,t),f \rangle|\nonumber\\
\leq&|\langle S_{t}^{N},f-f\ast\phi_{N}^{r}\rangle|+|\langle S_{t}^{N}\ast\phi_{N}^{r}-\varrho(.,t),f \rangle|\nonumber\\
\leq&|\langle S_{t}^{N},1\rangle|\left\Vert f-f\ast\phi_{N}^{r}\right\Vert_{\infty}+|\langle S_{t}^{N}\ast\phi_{N}^{r}-\varrho(,t),f \rangle|\nonumber\\
\leq&\left\Vert f-f\ast\phi_{N}^{r}\right\Vert_{\infty}+\left\Vert f\right\Vert_{L^{2}}\left\Vert S_{t}^{N}\ast\phi_{N}^{r}-\varrho(.,t)\right\Vert_{L^{2}}\nonumber\\
\leq&CN^{-\beta/d}\left\Vert\nabla f\right\Vert_{\infty}+\left\Vert f\right\Vert_{L^{2}}\left\Vert S_{t}^{N}\ast\phi_{N}^{r}-\varrho(.,t)\right\Vert_{L^{2}}.
\end{align}
From (\ref{limXNvarr})  we deduce 
\begin{equation*}
 \EE  \|S_{t\wedge \tau_{m}}^{N}- \varrho_{t\wedge \tau_{m}} \|_{-\alpha}^{2} \leq  C_{m,T} (N^{-\beta/ d} + \EE Q_{0}^{N}). \end{equation*}
For  $g\in H^{\alpha}(  \mathbb{R}^{d}) $ we obtain 
\begin{align*}
 &\left|\left\langle V_{t}^{N}-\varrho(.,t)\upsilon(.,t),g\right\rangle\right|\\
=&\left|\frac{1}{N}\sum_{k=1}^{N}V_{t}^{k,N}\cdot g(X_{t}^{k,N})-\left\langle\varrho(.,t),\upsilon(.,t)\cdot g\right\rangle\right|\\
=&\left|\frac{1}{N}\sum_{k=1}^{N}(V_{t}^{k,N}-\upsilon(X_{t}^{k,N}),t)\cdot g(X_{t}^{k,N})+\left\langle S_{t}^{N}-\varrho(.,t),\upsilon(.,t)\cdot g\right\rangle\right|\\  
\leq&\frac{1}{N}\sum_{k=1}^{N}\left|V_{t}^{k,N}-\upsilon(X_{t}^{k,N},t)\right|\left| g(X_{t}^{k,N})\right|+\left|\left\langle S_{t}^{N}-\varrho(.,t),\upsilon(.,t)\cdot g\right\rangle\right|\\
\leq&\Vert g\Vert_{\infty}\frac{1}{N}\sum_{k=1}^{N}\left|V_{t}^{k,N}-\upsilon(X_{t}^{k,N},t)\right|+|\langle S_{t}^{N}-\varrho(.,t),\upsilon(.,t)\cdot g\rangle|
\end{align*}
\begin{align}\label{convV}
\leq& C\Vert g\Vert_{\infty}\left(\frac{1}{N}\sum_{k=1}^{N}|V_{t}^{k,N}-\upsilon(X_{t}^{k,N},t)|^{2} \right)^{1/2}+|\langle S_{t}^{N}-\varrho(.,t),\upsilon(.,t)\cdot g\rangle|.\nonumber\\
&
\end{align}

From  (\ref{convV}) we have 
\begin{equation*}
  \EE \|V_{t\wedge\tau_{m}}^{N}- (\varrho \upsilon)_{t\wedge \tau_{m}}\|_{-\alpha}^{2} \leq  C_{m,T} (N^{-\beta/d} + \EE Q_{0}^{N}).  
\end{equation*}
\end{proof}
	
The proof of the following result  is also present  in \cite{correa}.
\begin{lemma}\label{Ito-correction} We have 
  \begin{align}\label{cov1}
&-\dfrac{4}{N}\sum_{k=1}^{N} \upsilon_{q}(X_{t}^{k,N},t) \sigma_{q}(X_{t}^{k,N}) V_{t,q}^{k,N}\circ dB_{t}^{q}\nonumber\\
&=-\dfrac{4}{N}\sum_{k=1}^{N} \upsilon_{q}(X_{t}^{k,N},t) \sigma_{q}(X_{t}^{k,N}) V_{t,q}^{k,N} dB_{t}^{q}-  \dfrac{4}{N}\sum_{k=1}^{N} \upsilon_{q}(X_{t}^{k,N},t) |\sigma_{q}(X_{t}^{k,N})|^{2} V_{t,q}^{k,N} dt,\nonumber\\
\end{align}
\begin{align}\label{cov2}
&\dfrac{2}{N}\sum_{k=1}^{N} V_{t,q}^{k,N}  \sigma_{q}(X_{t}^{k,N}) V_{t,q}^{k,N} \circ d B_{t}^{q}\nonumber\\
&=\dfrac{2}{N}\sum_{k=1}^{N} V_{t,q}^{k,N}  \sigma_{q}(X_{t}^{k,N}) V_{t,q}^{k,N}  d B_{t}^{q}+ \dfrac{2}{N}\sum_{k=1}^{N} |V_{t,q}^{k,N}|^{2}  |\sigma_{q}(X_{t}^{k,N})|^{2}  dt,
\end{align}
and 
\begin{align}\label{cov3}
&\dfrac{2}{N}\sum_{k=1}^{N}  \upsilon_{q}(X_{t}^{k,N},t) \sigma_{q}(X_{t}^{k,N}) \upsilon_{q}(X_{t}^{k,N},t)  \circ dB_{t}^{q}\nonumber\\
&=\dfrac{2}{N}\sum_{k=1}^{N} \upsilon_{q}(X_{t}^{k,N},t) \sigma_{q}(X_{t}^{k,N}) \upsilon_{q}(X_{t}^{k,N},t) dB_{t}^{q}+ \dfrac{2}{N}\sum_{k=1}^{N} |\upsilon_{q}(X_{t}^{k,N},t)|^{2} |\sigma_{q}(X_{t}^{k,N})|^{2} dt. 
\end{align}
\end{lemma}

\begin{proof} We will verifies (\ref{cov1}). 
The relation between Ito and Stratonovich integrals is (see \cite{Ku2})
\begin{align*}
& \dfrac{2}{N}\sum_{k=1}^{N} \upsilon_{q}(X_{t}^{k,N},t) \sigma_{q}(X_{t}^{k,N}) V_{t,q}^{k,N}\circ dB_{t}^{q} \\  
=&\dfrac{2}{N}\sum_{k=1}^{N} \upsilon_{q}(X_{t}^{k,N},t) \sigma_{q}(X_{t}^{k,N}) V_{t,q}^{k,N}  dB_{t}^{q} + 
\frac{1}{2}  \left[  \dfrac{2}{N}\sum_{k=1}^{N} \upsilon_{q}(X_{.}^{k,N},.) \sigma_{q}(X_{.}^{k,N}) V_{.,q}^{k,N}, B_{.}  \right]_{t}
\end{align*}
where $ \left[.,. \right]_{t} $ denotes the joint quadratic variation. Applying Ito  formula for the product,  Ito-Kunita-Wentzell formula   and  considering the  expression  of   $d[V_{t,q}^{k,N}\upsilon_{q}(X_{t}^{k,N},t)]$ given in  \eqref{went}, we 
have 
\begin{align}\label{upsiqVqrhoqN}
&d\left(-\dfrac{4}{N}\sum_{k=1}^{N} \upsilon_{q}(X_{t}^{k,N},t) \sigma_{q}(X_{t}^{k,N}) V_{t,q}^{k,N}\right)\nonumber\\
=&-\dfrac{4}{N}\sum_{k=1}^{N}\sigma_{q}(X_{t}^{k,N})\circ d[V_{t,q}^{k,N}\upsilon_{q}(X_{t}^{k,N},t)]-\dfrac{4}{N}\sum_{k=1}^{N} \upsilon_{q}(X_{t}^{k,N},t) V_{t,q}^{k,N}\circ d(\sigma_{q}(X_{t}^{k,N}))\nonumber\\
=&\dfrac{4}{N}\sum_{k=1}^{N}\sigma_{q}(X_{t}^{k,N})\upsilon_{q}(X_{t}^{k,N},t)\nabla^{q}\left( S_{t}^{N}\ast\phi_{N}\right)(X_{t}^{k,N})dt\nonumber\\
&+\dfrac{2}{N^{2}}\sum_{l,k=1}^{N}\sigma_{q}(X_{t}^{k,N})\upsilon_{q}(X_{t}^{k,N},t)\zeta_{N,q,\ast}\left(X_{t}^{k,N}-X_{t}^{l,N}\right)\left(V_{t}^{k,N}-V_{t}^{l,N}\right)dt\nonumber\\
&+\dfrac{4}{N}\sum_{k=1}^{N}\sigma_{q}(X_{t}^{k,N})V_{t,q}^{k,N}\upsilon(X_{t}^{k,N},t)\cdot\nabla\upsilon_{q}(X_{t}^{k,N},t)dt\nonumber\\
&+\dfrac{4}{N}\sum_{k=1}^{N}\sigma_{q}(X_{t}^{k,N})V_{t,q}^{k,N}\nabla^{q}\varrho(X_{t}^{k,N},t)dt\nonumber\\
&-\dfrac{4}{N}\sum_{k=1}^{N}\frac{1}{2}\sigma_{q}(X_{t}^{k,N})V_{t,q}^{k,N}\nabla^{q}\left(\varrho^{2}\operatorname{div}_{x} \upsilon\right)(X_{t}^{k,N},t)dt\nonumber\\
&-\dfrac{4}{N}\sum_{k=1}^{N}\sigma_{q}(X_{t}^{k,N})V_{t,q}^{k,N}\frac{1}{2\varrho(X_{t}^{k,N},t)} \sum_{i=1}^{d} \nabla^{i}\left(\varrho^{2}\left[\nabla^{i} \upsilon_{q}+\nabla^{q} \upsilon_{i}\right]\right)(X_{t}^{k,N},t)dt\nonumber\\
&-\dfrac{4}{N}\sum_{k=1}^{N}\sigma_{q}(X_{t}^{k,N})V_{t,q}^{k,N}\nabla\upsilon_{q}(X_{t}^{k,N},t)\cdot V_{t}^{k,N}dt\nonumber\\
&-\dfrac{8}{N}\sum_{k=1}^{N}V_{t,q}^{k,N}|\sigma_{q}(X_{t}^{k,N})|^{2}\upsilon_{q}(X_{t}^{k,N},t)  \circ dB_{t}^{q}\nonumber\\
&-\dfrac{4}{N}\sum_{k=1}^{N} \upsilon_{q}(X_{t}^{k,N},t) V_{t,q}^{k,N}\nabla\sigma_{q}(X_{t}^{k,N})\cdot V_{t}^{k,N}dt.
\end{align}
Only the martingale part of  $-\dfrac{4}{N}\sum_{k=1}^{N} \upsilon_{q}(X_{t}^{k,N},t) \sigma_{q}(X_{t}^{k,N}) V_{t,q}^{k,N}$
counts in the joint quadratic variation.   From  \eqref{upsiqVqrhoqN}  the only term contributing to the covariance is

\begin{align*}
    &\left[-\dfrac{4}{N}\sum_{k=1}^{N} \upsilon_{q}(X_{t}^{k,N},t) \sigma_{q}(X_{t}^{k,N}) V_{t,q}^{k,N},B_{t}^{q}\right]\\
    &=\left[-\dfrac{8}{N}\sum_{k=1}^{N}\int_{0}^{t}V_{s,q}^{k,N}|\sigma_{q}(X_{s}^{k,N})|^{2}\upsilon_{q}(X_{s}^{k,N},s)  \circ dB_{s}^{q},B_{t}^{q}\right]\\
    &=-\dfrac{8}{N}\sum_{k=1}^{N}\int_{0}^{t}V_{s,q}^{k,N}|\sigma_{q}(X_{s}^{k,N})|^{2}\upsilon_{q}(X_{s}^{k,N},s)  ds.
\end{align*}

Then from   relation between the Stratonovich  and  Ito integral, we can deduce (\ref{cov1}).  Using the arguments we deduce 
 (\ref{cov2}) and (\ref{cov3}).

\end{proof}

\section*{Acknowledgements}
C. Olivera  is partially supported by  FAPESP-ANR
by  the grant Stochastic and Deterministic Analysis for Irregular Models$-2022/03379-0$, 
by FAPESP by the grant  $2020/04426-6$, and  CNPq by the grant $422145/2023-8$.
Author Jesus Correa  has received research grants from CNPq
through the grant  141464/2020-8.

\end{document}